\documentclass[envcountsame,envcountsect]{svjour3}
\usepackage{fix-cm}
\usepackage{amsmath}
\usepackage{amssymb}

\smartqed
\spnewtheorem{remark}[theorem]{Remark}{\bfseries}{\rm}
\spnewtheorem{example}[theorem]{Example}{\bfseries}{\rm}

\def\im{\operatorname{Im}}     \def\Ker{\operatorname{Ker}}
   \def\ri{\operatorname{ri}}
\def\Supp{\operatorname{Supp}} \def\Ann{\operatorname{Ann}}
\def\HF{\operatorname{HF}}     \def\HP{\operatorname{HP}}
\newcommand{\bbQ}{\ensuremath{\mathbb Q}}
\newcommand{\bbZ}{\ensuremath{\mathbb Z}}
\newcommand{\bbN}{\ensuremath{\mathbb N}}
\newcommand{\bbP}{\ensuremath{\mathbb P}}
\newcommand{\bbX}{\ensuremath{\mathbb X}}
\newcommand{\bbY}{\ensuremath{\mathbb Y}}
\newcommand{\bbW}{\ensuremath{\mathbb W}}
\newcommand{\bbV}{\ensuremath{\mathbb V}}

\def\tsum#1#2{{\textstyle\sum\limits_{#1}^{#2}}}
\def\tfrac#1#2{{\textstyle\frac{#1}{#2}}}

\begin{document}

\title{Hilbert Polynomials of K{\"a}hler Differential Modules 
for Fat Point Schemes}

\author{Martin Kreuzer \and Tran N.K. Linh \and Le~Ngoc~Long}

\institute{
Martin Kreuzer \at Fakult\"{a}t f\"{u}r Informatik und Mathematik,
Universit\"{a}t Passau, D-94030 Passau, Germany,\\
\email{martin.kreuzer@uni-passau.de}
\and
Tran N.K. Linh \at Department of Mathematics,
University of Education, Hue University, 34 Le Loi, Hue, Vietnam,
\email{tnkhanhlinh@hueuni.edu.vn}
\and
Le Ngoc Long \at Fakult\"{a}t f\"{u}r Informatik und Mathematik 
Universit\"{a}t Passau, D-94030 Passau, Germany 
and Department of Mathematics, University of Education, Hue University,
34 Le Loi, Hue, Vietnam,
\email{lelong@hueuni.edu.vn}
}

\maketitle



\medskip


\begin{abstract} Given a fat point scheme $\mathbb{W}=m_1P_1+\cdots+m_sP_s$ 
in the projective $n$-space~$\mathbb{P}^n$ over a field~$K$ of characteristic zero,
the modules of K\"ahler differential $k$-forms of its homogeneous
coordinate ring contain useful information about algebraic and
geometric properties of~$\mathbb{W}$ when $k\in\{1,\dots, n+1\}$. 
In this paper we determine the
value of its Hilbert polynomial explicitly for the case $k=n+1$, 
confirming an earlier conjecture. More precisely this value is given by the multiplicity
of the fat point scheme $\bbY = (m_1-1)P_1 + \cdots + (m_s-1)P_s$.
For $n=2$, this allows us to determine the Hilbert polynomials of the
modules of K\"ahler differential $k$-forms for $k=1,2,3$, and
to produce a sharp bound for the regularity index for $k=2$.
\end{abstract}

\bigbreak
\section{Introduction}

Let $\bbX = \{P_1,\dots,P_s\}$ be a set of points in projective
$n$-space~$\bbP^n$ over a field~$K$ of characteristic zero, and let $I_\bbX$ 
be the homogeneous vanishing ideal of~$\bbX$ in $S=K[X_0, \dots, X_n]$.
One important reason why the Hilbert function of~$I_\bbX$ has been studied
extensively is that the elements of $(I_\bbX)_d$ control the non-uniqueness
of the solution of the homogeneous polynomial interpolation problem in degree~$d$.
When we switch to the Hermite interpolation problem, this non-uniqueness
is controlled by the set of all polynomials having the property that not only
their values, but also their derivatives up to some order~$m_i$ vanish at the
point~$P_i$ for $i=1,\dots,s$. In the language of algebraic geometry, this means
that we are interested in the Hilbert function of the homogeneous vanishing
ideal $I_\bbW = I_{P_1}^{m_1} \cap \cdots \cap I_{P_s}^{m_s}$
of the {\it fat point scheme} $\bbW = m_1 P_1 + \cdots + m_s P_s$ in~$\bbP^n$.
Such schemes have arisen in several other contexts, for instance
in the study of projective varieties which are obtained from blowing up
sets of points in~$\bbP^n$ (see~\cite{CH2013}).

The Hilbert functions of the ideals~$I_\bbW$, or equivalently, of the homogeneous
coordinate rings $R_\bbW = S/I_\bbW$ have undergone intense scrutiny in the past.
A major breakthrough was the paper~\cite{AH1995} in which Alexander and Hirshowitz
determined the Hilbert function of fat point schemes consisting of double
points (i.e., with $m_1 = \cdots = m_s =2$) for the case of a generic support~$\bbX$.
For this, they used a local differential method which they called 
{\it La methode d{\!\'{}}Horace}. It is therefore a natural approach to study also global
differentials for fat point schemes. They are given by the module of K\"ahler
differentials $\Omega^1_{R_\bbW/K}$ and its exterior powers, the modules
of K\"ahler differential $k$-forms $\Omega^k_{R_\bbW/K}$ for $k\ge 1$.
In~\cite{DK1999}, De Dominicis and the first author initiated a careful 
examination of the structure and the Hilbert function of~$\Omega^1_{R_\bbX/K}$
for a set of points~$\bbX$ (i.e., for the case $m_1 = \cdots = m_s=1$).
Then, in~\cite{KLL2015}, the present authors started the study of the modules
of K\"ahler differential $k$-forms for fat point schemes, and in~\cite{KLL2018}
this work was continued. 

A major open question in~\cite{KLL2018} is a formula for the Hilbert polynomial
of $\Omega^{n+1}_{R_\bbW/K}$, i.e., for the value of the Hilbert function in large
degrees, which is given as a conjecture there. Our main result here is to prove this
formula. More precisely, in Theorem~\ref{ThmS2_7} we show that, for a
fat point scheme $\bbW=m_1 P_1 + \cdots + m_sP_s$, the Hilbert polynomial
of~$\Omega^{n+1}_{R_\bbW/K}$ is equal to the multiplicity of the {\it slimming}
$\bbY = (m_1-1)P_1 + \cdots + (m_s-1)P_s$ of~$\bbX$, and is therefore given by  
$\sum_{i=1}^s \binom{m_i+n-2}{n}$.

To achieve this goal, we proceed as follows. After recalling the definition and some
basic results for fat point schemes in Section~2, we compare the Hilbert functions
of a fat point scheme $\bbW=m_1 P_1 + \cdots +m_s P_s$, its support
$\bbX = P_1 + \cdots P_s$, and its slimming $\bbY = (m_1-1) P_1 + \cdots + 
(m_s-1) P_s$. More precisely, we show in Proposition~\ref{PropS1_6} that the Hilbert 
functions of~$I_\bbW$ and of $I_\bbX \cdot I_\bbY$ agree in sufficiently large degrees.
In the last part of this section we recall the definition of the module of
K\"ahler differential $k$-forms of~$R_\bbW/K$ and recall some of its basic properties.

Section 3 contains the main results of this paper. First we show that,
for a set of points~$\bbX$ in~$\bbP^n$ and a subset~$\bbY$ of~$\bbX$,
the vanishing ideals satisfy $(I_\bbX^k \cdot I_\bbY^\ell) _i \subseteq
(\partial(I_\bbX^{k+1}\cdot I_\bbY^\ell))_i$ for all $k,\ell\ge 0$ 
in sufficiently large degrees $i\gg 0$ (see Lemma~\ref{LemS2_4}). 
Then we use a subtle induction argument to prove
that, for a sequence of point sets $\bbY_1 \supseteq \bbY_2 \supseteq \cdots
\supseteq \bbY_t$ in~$\bbP^n$ and for $\nu_1,\dots,\nu_t>0$, we have
$(I_{\bbY_1}^{\nu_1} \cdots I_{\bbY_t}^{\nu_t})_i \subseteq
(\partial (I_{\bbY_1}^{\nu_1+1}\, I_{\bbY_2}^{\nu_2} \cdots I_{\bbY_t}^{\nu_t}))_i$
in sufficiently large degrees $i\gg 0$ (see Proposition~\ref{PropS2_5}). 
Using the presentation $\Omega^{n+1}_{R_\bbW/K} \cong (S/\partial I_\bbW)(-n-1)$ 
and starting from the equimul\-tiple case $m_1=\cdots = m_s$, we then prove
inductively Theorem~\ref{ThmS2_7} which says that the Hilbert function
of the module of K\"ahler $(n+1)$-forms of a fat point scheme 
$\bbW=m_1 P_1 + \cdots + m_s P_s$ is given in large degrees
by the Hilbert function of its slimming~$\bbY$. On other words, we have
$\HP(\Omega^{n+1}_{R_\bbW/K}) = \sum_{i=1}^s \binom{m_i+n-2}{n}$.

The final section is devoted to applying this theorem in the case of
fat point schemes in the plane~$\bbP^2$. In this case we determine the Hilbert
polynomials of $\Omega^k_{R_\bbW/K}$ for all three relevant cases $k=1,2,3$
(see Proposition~\ref{PropS3_1}) and provide a bound from where on the
canonical exact sequence
$$
\begin{aligned}
0\rightarrow (I_{\bbW^{(1)}}/I_{\bbW^{(2)}})_i
&\stackrel{\alpha}{\longrightarrow}
(I_\bbW\Omega^1_{S/K}/I_{\bbW^{(1)}}\Omega^1_{S/K})_i \\
&\stackrel{\beta}{\longrightarrow}
(\Omega^2_{S/K}/ I_\bbW\Omega^2_{S/K})_i
\stackrel{\gamma}{\longrightarrow}
(\Omega^2_{R_\bbW/K})_i\rightarrow 0.
\end{aligned}
$$
is exact in degree~$i$ (see Proposition~\ref{PropS3_3}). Here $\bbW^{(1)}$
and $\bbW^{(2)}$ are the first and second {\it fattenings}
$\bbW^{(j)} = (m_1+j)P_1 + \cdots + (m_s+j)P_s$ of~$\bbW$.
Using an explicit example, we certify that this result yields a sharp 
bound for the regularity index of~$\Omega^2_{R_\bbW/K}$.

The research underlying this paper and the calculation of the examples were
greatly aided by an implementation of the relevant objects and functions
in the computer algebra system ApCoCoA (see~\cite{ApC}). Unless explicitly 
stated otherwise, we adhere to the definitions and notation introduced in
the books~\cite{KR2000} and~\cite{KR2005}, as well as in our previous 
papers~\cite{KLL2015} and~\cite{KLL2018}.

%
%

\bigbreak
\section{K{\"a}hler Differentials for Fat Point Schemes}\label{sec2}

Throughout this paper we work over a field~$K$ of characteristic zero,
and we let $S=K[X_0,...,X_n]$ be a standard graded polynomial ring
over~$K$. By~$\mathfrak{M}$ we denote the
homogeneous maximal ideal $\langle X_0, \dots, X_n \rangle$ of~$S$.
The ring $S$ is the homogeneous coordinate ring of the projective
$n$-space $\bbP^n$ over $K$.

Let $P_1,\dots,P_s$ be distinct $K$-rational points in $\mathbb{P}^n$. 
The prime ideals in~$S$ corresponding to the points $P_1,\dots,P_s$
will be denoted by $I_{P_1},\dots, I_{P_s}$, respectively.

\begin{definition} Let $m_1,\dots,m_s$ be positive integers.
\begin{enumerate}
\item[(a)] A zero-dimensional scheme $\bbW$ in $\bbP^n$ is called a 
\textit{fat point scheme} if it is defined by a saturated ideal
of the form $I_{\mathbb{W}}=I_{P_1}^{m_1}\cap \cdots \cap I_{P_s}^{m_s}$.
In this case, we also write
$$
\bbW \;=\; m_1P_1 + \cdots + m_sP_s.
$$

\item[(b)] The number $m_j$ is called the \textit{multiplicity} of the point
$P_j$ for $j=1,\dots,s$.

\item[(c)] If $m_1=\dots =m_s =: m$, we refer to $\bbW$ as an
\textit{equimultiple fat point scheme} and denote it also by $m\mathbb{X}$.

\item[(d)] The set of points $\Supp(\bbW) := \{P_1,\dots,P_s\}$ is called 
the \textit{support} of~$\bbW$. Subse\-quently, we also write 
$\Supp(\bbW)=P_1+\cdots+P_s$.
\end{enumerate}
\end{definition}

It is well known that the homogeneous coordinate ring $R_\bbW:=S/I_\bbW$
of $\bbW$ is a one-dimensional, Cohen-Macaulay, standard graded $K$-algebra.
The Hilbert function $\HF_{R_\bbW}(i):=\dim_K(R_\bbW)_i$ of $R_\bbW$
will be denoted by $\HF_\bbW$. Note that $\HF_\bbW$ is strictly increasing
until it reaches the degree $\deg(\bbW)=\sum_{i=1}^s\binom{m_i+n-1}{n}$
of~$\bbW$ at which it stabilizes. The least integer $i$ for which
$\HF_{\bbW}(i)=\deg(\bbW)$ is called the \textit{regularity index}
of $\HF_\bbW$ and is denoted by $r_\bbW$. In this setting, the Hilbert polynomial
of $\HF_\bbW$ is the constant polynomial given by $\HP_\bbW(z)=\deg(\bbW)$.
\medskip

\begin{definition}~\label{DefnS1_2}
Let $\bbW = m_1P_1+\cdots+m_sP_s$ be a fat point scheme in $\bbP^n$,
let $j\in \{1,\dots,s\}$, and let $\bbW_j\subseteq \bbW$ be the
fat point scheme
$$
\bbW_j = m_1P_1+\cdots+m_{j-1}P_{j-1}+ (m_j-1)P_j
+m_{j+1}P_{j+1}+\cdots+m_sP_s
$$
obtained by reducing the multiplicity of $P_j$ by one.
If $m_j=1$, then $P_j$ does not
appear in the support of $\bbW_j$. Further, let
$\nu_j = \deg(\bbW)-\deg(\bbW_j)$.

\begin{enumerate}
\item[(a)] An element $F\in I_{\bbW_j}\setminus I_\bbW$ is called
a \textit{separator} of $\bbW_j$ in $\bbW$.

\item[(b)] A set of homogeneous polynomials $\{ F_1,\dots,F_t \}$
is called a \textit{minimal set of separators} of $\bbW_j$ in $\bbW$
if $t=\nu_j$ and $I_{\bbW_j} = I_\bbW+\langle F_1,\dots,F_t \rangle$.
\end{enumerate}
\end{definition}

According to \cite[Theorem~3.3]{GMT2010}, a minimal set of separators
of $\bbW_j$ in $\bbW$ always exists. Some basic properties of 
separators are described by
the following lemma which follows from~\cite[Lemma~5.1]{GMT2010}.

\begin{lemma}\label{LemS1_3}
In the setting of Definition~\ref{DefnS1_2}, let $\{F_1,...,F_{\nu_j}\}$
be a minimal set of separators of~$\bbW_j$ in $\bbW$, and suppose that
$\deg(F_1)\le \cdots \le \deg(F_{\nu_j})$.

\begin{enumerate}
\item[(a)] For $k=1,\dots,\nu_j$, we have
$(I_\bbW+\langle F_1,\dots,F_{k-1}\rangle):\langle F_k\rangle = I_{P_j}$.

\item[(b)] For $k=1,\dots,\nu_j$, the ideal
$I_\bbW+\langle F_1,\dots,F_k\rangle$ is a saturated ideal.
\end{enumerate}
\end{lemma}

By using separators, we want to figure out a connection between
the homoge\-neous vanishing ideals of $\bbW$, of the support
$\bbX:=\Supp(\bbW)$, and of the fat point scheme
$\mathbb{Y}:=(m_1-1)P_1+\cdots+(m_s-1)P_s$ which we call the {\it slimming}
of~$\bbW$. First we have the following relation between the intersection
and the product of two homogeneous ideals in~$S$.

\begin{lemma} \label{LemS1_4}
Let $I, J$ be two homogeneous ideals of $S$.
Then the following state\-ments hold true.

\begin{enumerate}
\item[(a)] The graded module $M=(I\cap J)/(I \cdot J)$ is annihilated by $I+J$.

\item[(b)] If the Hilbert polynomial of $S/(I+J)$ satisfies
$\HP_{S/(I+J)}(z)=0$, then we have the equality
$\HP_{S/(I\cap J)}(z) = \HP_{S/(I\cdot J)}(z)$.
In particular, we have $(I\cap J)_i=(I \cdot J)_i$ for $i\gg 0$.
\end{enumerate}
\end{lemma}

\begin{proof}
Claim (a) follows from the fact that
for $f\in I$, $g\in J$ and $h\in I\cap J$ we have
$(f+g)h = fh + gh \in I\cdot J$.

Now we prove (b). Since we have $\HP_{S/(I+J)}(z)=0$,
there exists $i_0\in \bbN$ such that
$(I+J)_i=\mathfrak{M}_i$ for all $i\ge i_0$.
It follows from (a) that $I+J\subseteq \Ann_S(M)$, and hence
$(\Ann_S(M))_i=\mathfrak{M}_i$ for all $i\ge i_0$.
This implies $\dim(S/\Ann_S(M))=0$.
Using \cite[Theorem 5.4.10]{KR2005}, we then get
$\dim(M)=\dim(S/\Ann(M))=0$. Consequently, by \cite[Theorem 5.4.15]{KR2005},
the Hilbert polynomial of~$M$ is $\HP_M(z)=0$. Thus~(b) follows from the 
homogeneous exact sequence
\begin{equation*}
0 \;\longrightarrow\; M \;\longrightarrow\; S/(I\cdot J) \;\longrightarrow\;
S/(I\cap J) \;\longrightarrow\; 0 \tag*{$\square$}
\end{equation*}
\end{proof}

\begin{remark}\label{RemS1_5}
Let $\bbW = m_1P_1+\cdots + m_sP_s$ and
$\bbV = m'_1P'_1+\cdots+m'_tP'_t$
be fat point schemes in $\mathbb{P}^n$ such that
$\Supp(\bbW)\cap\Supp(\bbV)=\emptyset$.
According to \cite[Proposition 5.4.16]{KR2005}, we have
$$
\HP_{S/(I_\bbW\cap I_\bbV)}(z) = \HP_{S/I_\bbW}(z)+\HP_{S/I_\bbV}(z)
- \HP_{S/(I_\bbW+I_\bbV)}(z).
$$
The assumption yields that
$\HP_{S/(I_\bbW\cap I_\bbV)}(z) = \HP_{S/I_\bbW}(z)+\HP_{S/I_\bbV}(z)$,
and so we get $\HP_{S/(I_\bbW+I_\bbV)}(z)=0$.
In this case the lemma implies
$\HP_{S/(I_\bbW\cap I_\bbV)}(z) =\HP_{S/ (I_\bbW \cdot I_\bbV)}(z)$.
\end{remark}

\begin{proposition}\label{PropS1_6}
Let $\bbW = m_1P_1+\cdots+m_sP_s$ be a fat point scheme
supported at $\bbX$ in $\mathbb{P}^n$, and let $\bbY$ be
the slimming $\bbY = (m_1-1)P_1+\cdots+(m_s-1)P_s$ of $\bbW$.

\begin{enumerate}
\item[(a)] We have $I_\bbW :_S I_\bbY = I_\bbX$.

\item[(b)] There exists $i_0\in \mathbb{N}$ such that
for all $i\ge i_0$ we have $(I_\bbW)_i=(I_\bbX \cdot I_\bbY)_i$.
\end{enumerate}
\end{proposition}

\begin{proof} (a)\quad
Clearly, we have $I_\bbX\cdot I_\bbY\subseteq I_\bbW$,
and so $I_\bbX\subseteq I_\bbW:_S I_\bbY$. For the other inclusion,
let $\nu_j=\binom{m_j+n-2}{n-1}$ and let
$\{F_{j1},...,F_{j\nu_j}\}$ be a minimal set of separators
of~$\bbW_j$ in $\bbW$ such that
$\deg(F_{j1})\le\cdots\le \deg(F_{j\nu_j})$, where
$\bbW_j = m_1P_1+\cdots+m_{j-1}P_{j-1}+ (m_j-1)P_j
+m_{j+1}P_{j+1}+\cdots+m_sP_s$.
Then, by Lemma~\ref{LemS1_3}, we have
$$
I_\bbY = I_\bbW + \langle F_{11},\dots,F_{1\nu_1},\dots,
F_{s1},\dots,F_{s\nu_s}\rangle.
$$
Suppose for a contradiction that there exists a homogeneous
element $F\in (I_\bbW :_S I_\bbY)\setminus I_\bbX$.
Since $F\notin I_\bbX$, there exists $j\in\{1,\dots,s\}$
such that $F \notin I_{P_j}$.
On the other hand, we have $F\cdot F_{jk}\in I_\bbW$ for all $j=1,\dots,s$ and $k=1,\dots,\nu_j$.
In particular, we get $F \in I_\bbW :_S \langle F_{j1}\rangle$.
Also, by Lemma~\ref{LemS1_3}, the separator $F_{j1}$ satisfies
$I_\bbW :_S \langle F_{j1}\rangle = I_{P_j}$.
Hence we obtain $F \in I_{P_j}$, a contradiction.

\medskip (b)\quad
Note that $(I_\bbX\cdot I_\bbY)_i\subseteq (I_\bbW)_i$
for all $i\in\bbN$. Set $\bbV := m_1P_1+\cdots+m_{s-1}P_{s-1}$.
It follows from Remark~\ref{RemS1_5} that
$$
\HP_{S/(I_\bbV + I_{m_sP_s})}(z)=0.
$$
An application of Lemma~\ref{LemS1_4} yields that
there exists $t\in \bbN$ such that for all $i\ge t$
we have $(I_\bbW)_i =  (I_\bbV \cdot I_{m_sP_s})_i$.
By induction on $s$ we find $i_0\in\bbN$ such that
\begin{align*}
(I_\bbW)_i &=  (I_\bbV \cdot I_{m_sP_s})_i
  =(I_{m_1P_1}\cdots I_{m_{s-1}P_{s-1}}\cdot I_{m_sP_s})_i \\
  &\stackrel{(\ast)}{=} (I_{P_1}^{m_1}\cdots I_{P_s}^{m_s})_i \\
  &= (I_{P_1}\cdots I_{P_s}\cdot I_{P_1}^{m_1-1}\cdots I_{P_s}^{m_s-1})_i
\end{align*}
for all $i\ge i_0$, where the equality $(\ast)$ follows from
the fact that $I_{P_j}$ is a complete intersection ideal
for $j=1,...,s$. Moreover, observe that
$I_{P_1}\cdots I_{P_s}\subseteq I_\bbX$ and
$I_{P_1}^{m_1-1}\cdots I_{P_s}^{m_s-1} \subseteq I_\bbY$.
This implies $(I_\bbW)_i \subseteq (I_\bbX\cdot I_\bbY)_i$,
and therefore we get the equality
$(I_\bbW)_i =(I_\bbX\cdot I_\bbY)_i$.
\qed\end{proof}

Notice that if $\bbW$ is an equimultiple fat point scheme
in $\bbP^n$ such that its support $\bbX$ is a complete intersection,
then claim (b) of Proposition~\ref{PropS1_6} holds true
for $i_0=0$. However, when $\bbW$ is not
an equimultiple fat point scheme
and~$\bbX$ is a complete intersection,
the following example shows that the number $i_0$
in Proposition~\ref{PropS1_6}.b should be chosen large enough.

\begin{example} \label{ExamS1_1}
Let $\bbX\subseteq \bbP^2$ be the scheme
$\bbX = P_1+P_2+\cdots+P_8$ consisting of eight points
given by $P_1=(1:0:0)$, $P_2=(1:0:1)$, $P_3=(1:1:0)$, $P_4=(1:1:1)$,
$P_5=(1:2:0)$, $P_6=(1:2:1)$, $P_7=(1:3:0)$ and $P_8=(1:3:1)$.
The homogeneous vanishing ideal of $\bbX$ is given by
$$
I_\bbX =
\langle X_0X_2 -X_2^2,  6X_0^3X_1 -11X_0^2X_1^2 +6X_0X_1^3 -X_1^4 \rangle
\subseteq S=K[X_0,X_1,X_2],
$$
and so $\bbX$ is a complete intersection of type $(2,4)$.

Now we consider the fat point scheme $\bbW = P_1+2P_2+P_3+2P_4+2P_5+P_6+5P_7+P_8$
supported at~$\bbX$. Let the subscheme $\bbY$ of~$\bbW$ be
given by $\bbY=P_2+P_4+P_5+4P_7$.
A calculation using ApCoCoA (see~\cite{ApC}) yields the Hilbert functions
\begin{align*}
&\HF_\bbX:\ 1\ 3\ 5\ 7\ 8\ 8\ \cdots, &\!&
\HF_\bbY:\ 1\ 3\ 6\ 10\ 13\ 13\cdots, \\
&\HF_\bbW:\ 1\ 3\ 6\ 10\ 15\ 21\ 26\ 27\ 28\ 28\cdots, &\!&
\HF_{S/(I_\bbX\cdot I_\bbY)}\!:\ 1\ 3\ 6\ 10\ 15\ 21\ 26\ 28\ 28
\cdots\!.
\end{align*}
This implies that $\HF_\bbW(7) =27 < 28 = \HF_{S/(I_\bbX\cdot I_\bbY)}(7)$.
Hence we get $I_\bbW \ne I_\bbX \cdot I_\bbY$. But we may check that
$(I_\bbW)_i = (I_\bbX \cdot I_\bbY)_i$ for all $i\ge 8$.
\end{example}

The next example shows that, if the support of
an equimultiple fat point scheme $\bbW$ is not a complete intersection,
then Proposition~\ref{PropS1_6}.b does not always hold for all $i\in \bbN$.

\begin{example} \label{ExamS1_2}
Let us consider the set $\bbX'=\bbX\setminus\{P_4\}\subseteq \bbP^2$
where $\bbX$ is the complete intersection given in Example~\ref{ExamS1_1}.
Then $\bbX'$ is an almost complete intersection with
$$
I_{\bbX'} = I_\bbX +
\langle X_1^3X_2 -5X_1^2X_2^2 +6X_1X_2^3 \rangle
\subseteq S=K[X_0,X_1,X_2].
$$
The double point scheme
$\bbW' = 2P_1+2P_2+2P_3+2P_5+2P_6+2P_7+2P_8 \subseteq \bbP^2$
supported at $\bbX'$
has the Hilbert function $\HF_{\bbW'}:\ 1\ 3\ 6\ 10\ 14\ 18\ 20\ 21\ 21\cdots$.
In this case the Hilbert function of $S/I_{\bbX'}^2$ is given by
$\HF_{S/I_{\bbX'}^2}:\ 1\ 3\ 6\ 10\ 14\ 18\ 20\ 22\ 21\ 21\cdots$.
Therefore we obtain
$\HF_{\bbW'}(7) =21 < 22 = \HF_{S/I_{\bbX'}^2}(7)$, and hence
$I_{\bbW'}\ne I_{\bbX'}\cdot I_{\bbX'}$.
\end{example}

Now we introduce the algebraic object associated to $\bbW$
which we are most interested in.
The enveloping algebra of $R_\bbW$ is the graded algebra
$R_\bbW\otimes_K R_\bbW =
\bigoplus_{i\ge 0}(\bigoplus_{j+k=i} (R_\bbW)_j\otimes(R_\bbW)_k)$.
Let $\mu:R_\bbW \otimes_K R_\bbW \rightarrow R_\bbW$ be
the canonical multiplication map given by $\mu(f\otimes g)=fg$
for all $f,g\in R_\bbW$. This map is homogeneous of degree zero
and its kernel $\mathcal{J}:=\Ker(\mu)$ is a homogeneous ideal
of~$R_\bbW\otimes_K R_\bbW$.

\begin{definition} Let $k$ be a positive integer.
\begin{enumerate}
\item[(a)] The graded $R_\bbW$-module
$\Omega^1_{R_\bbW/K} :=\mathcal{J}/\mathcal{J}^2$
is called the \textit{module of K{\"a}hler differentials}
of $R_\bbW/K$. The homogeneous $K$-linear map
$d:R_\bbW\rightarrow \Omega^1_{R_\bbW/K}$ given by
$f\mapsto f\otimes 1-1\otimes f + \mathcal{J}^2$ 
is called the \textit{universal derivation} of~$R_\bbW/K$.

\item[(b)] The exterior power
$\Omega^k_{R_\bbW/K}:= \bigwedge_{R_\bbW}^k \Omega^1_{R_\bbW/K}$
is called the \textit{module of K{\"a}hler differential $k$-forms}
of $R_\bbW/K$.
\end{enumerate}
\end{definition}

For $i=0,\dots,n$, we denote the image of $X_i$ in~$R_\bbW$ by $x_i$.
Then we have $\deg(dx_i)=\deg(x_i)=1$ and
$\Omega^1_{R_\bbW/K} = R_\bbW dx_0+\cdots+R_\bbW dx_n$.
Hence we see that $\Omega^k_{R_\bbW/K}=0$ for all $k\ge n+2$, and 
$\Omega^k_{R_\bbW/K}$ is a finitely generated $R$-module for all $k\ge 1$.
Moreover, from \cite[Proposition 4.12]{Kun1986} or 
\cite[Proposition~3.2.11]{TNKL2015} we get the
following presentation for $\Omega^k_{R_\bbW/K}$.

\begin{proposition} \label{PropS1_10}
Let $1\le k\le n+1$ and let
$dI_\bbW = \langle dF \mid F\in I_\bbW \rangle$.
Then the graded $R_\bbW$-module $\Omega^k_{R_\bbW/K}$
has a presentation
$$
\Omega^k_{R_\bbW/K}\cong \Omega^k_{S/K}/
(I_\bbW \Omega^k_{S/K}+ dI_\bbW\Omega^{k-1}_{S/K}).
$$
\end{proposition}

Given a non-zero homogeneous ideal~$I$ of~$S$, we denote
the homogeneous ideal
$\langle \frac{\partial F}{\partial X_i} | F\in I, 0\le i\le n \rangle
\subseteq S$ by~$\partial I$.
The ideal $\partial I$ is also known as the
\textit{$n$-th Jacobian ideal} (or the \textit{$n$-th K\"ahler different})
of the $K$-algebra $S/I$ (see \cite[Section~10]{Kun1986}).
If $\{F_1,\dots,F_r\}$ is a set of generators of~$I$, then
$\partial I$ is generated by all $1$-minors of
the Jacobian matrix
$(\frac{\partial F_k}{\partial X_j})_{0\le j\le n,1\le k\le n}$,
i.e., $\partial I = \langle \frac{\partial F_k}{\partial X_j}
\mid 0\le j\le n, 1\le k\le r\rangle$.
Furthermore, by Euler's relation \cite[Section~1]{Kun1986},
we always have $I\subseteq \partial I$.

\begin{lemma}\label{LemS1_11}
Let $I$ and $J$ be two proper homogeneous ideal of $S$
such that $I_i=J_i$ for $i\gg 0$. Then we have
$(\partial I)_i=(\partial J)_i$ for $i\gg 0$.
\end{lemma}

\begin{proof}
It suffices to prove the inclusion
$(\partial I)_i\subseteq (\partial J)_i$ for $i\gg 0$,
since~$I$ and~$J$ may be interchanged.
Let $i_0\in\bbN$ be a number such that $I_i=J_i$ for all $i\ge i_0$.
For $i\ge i_0$ let $F\in (\partial I)_i$.
There are homogeneous polynomials $G_{jk}\in S$ and
$H_{jk}\in I$ such that
$F=\sum_{j=0}^n\sum_{k=1}^m G_{jk}\frac{\partial H_{jk}}{\partial X_j}$.
Then, for $j\in\{0,...,n\}$ and $k\in\{1,...,m\}$, we have
$G_{jk}\frac{\partial H_{jk}}{\partial X_j}=
\frac{\partial(G_{jk}H_{jk})}{\partial X_j} -
H_{jk}\frac{\partial G_{jk}}{\partial X_j}$.
Obviously, $G_{jk}H_{jk} \in I_{i+1}=J_{i+1}$, and hence
$\frac{\partial(G_{jk}H_{jk})}{\partial X_j} \in (\partial J)_{i}$.
Also, we have $J \subseteq \partial J$ and
$H_{jk}\frac{\partial G_{jk}}{\partial X_j}\in I_i
=J_i\subseteq (\partial J)_i$. Thus
$G_{jk}\frac{\partial H_{jk}}{\partial X_j}\in (\partial J)_i$.
Consequently, we get $F\in (\partial J)_i$, as we wanted to show.
\qed\end{proof}

As a consequence of Proposition~\ref{PropS1_10}, we obtain the following
explicit descrip\-tion for the module of K{\"a}hler differential ($n+1$)-forms
of $R_\bbW /K$ (see also \cite[Corollary~2.3]{KLL2018}).

\begin{corollary}\label{CorS1_12}
There is an isomorphism of graded $R_\bbW$-modules
$$
\Omega^{n+1}_{R_\bbW /K}\cong (S/\partial I_\bbW)(-n-1)
$$
In particular, we have $\HF_{\Omega^{n+1}_{R_\bbW /K}}(i)=\HF_{S/\partial I_\bbW}(i-n-1)$
for all $i\in\bbZ$.
\end{corollary}

%
%

\bigbreak
\section{Hilbert Polynomials of K{\"a}hler Differential Modules}%
\label{sec3}

In this section we look at the Hilbert function of the module of
K\"ahler differential $k$-forms of a fat point scheme
$\bbW = m_1P_1+\cdots+m_sP_s$ in $\bbP^n$, where $1\le k\le n+1$.
Especially, for the case $k=n+1$, we determine a formula
for the Hilbert polynomial of the module of K\"ahler differential
$k$-forms of~$\bbW$.

Clearly, the ring $R_\bbW$ is Noetherian and the module of
K\"ahler differential $k$-forms $\Omega^k_{R_\bbW/K}$ is
a finitely generated graded $R_\bbW$-module.
Thus the Hilbert poly\-nomial $\HP_{\Omega^k_{R_\bbW/K}}(z)\in \bbQ[z]$
of~$\Omega^k_{R_\bbW/K}$ exists (see e.g.~\cite[Theorem~4.1.3]{BH1993}).
The \textit{(Hilbert) regularity index}
of~$\Omega^k_{R_\bbW/K}$ is defined by
$\ri(\Omega^k_{R_\bbW/K}) := \min\{ i\in \bbZ \mid
\HF_{\Omega^k_{R_\bbW/K}}(j)=\HP_{\Omega^k_{R_\bbW/K}}(j)$
for all $j\ge i \}$.

The Hilbert polynomial of $\Omega^k_{R_\bbW/K}$ is easily shown
to be a constant polynomial. However,
except the case $n=1$ (see e.g. \cite{Rob1989}), 
to determine the Hilbert polynomial
of $\Omega^k_{R_\bbW/K}$ is an interesting non-trivial task.
In \cite[Sections~4 and 5]{KLL2018}, the authors
gave the following bounds for the Hilbert polynomial
$\HP_{\Omega^k_{R_\bbW/K}}(z)$ and its regularity index.

\begin{proposition}\label{PropS2_1}
Let $\bbW=m_1P_1+\cdots+m_sP_s$ be a fat point scheme
supported at $\bbX$ in~$\mathbb{P}^n$, and let $\bbV$ be
the fat point scheme $\bbV =(m_1+1)P_1+\cdots+(m_s+1)P_s$.
Then the Hilbert polynomial $\HP_{\Omega^k_{R_\bbW/K}}(z)$ satisfies
$$
\sum_{i=1}^s \tbinom{n+1}{k}\tbinom{m_i+n-2}{n}
\le \HP_{\Omega^k_{R_\bbW/K}}(z) \le
\sum_{i=1}^s \tbinom{n+1}{k}\tbinom{m_i+n-1}{n}
$$
and $\ri(\Omega^k_{R_\bbW/K}) \le \min\{\max\{r_\bbW + k, r_\bbV+k-1\},
\max\{r_\bbW + n, r_\bbV+n-1\}\}$.
In particular, for $\nu \ge 1$, we have
$\HP_{\Omega^{n+1}_{R_{\bbX}/K}}(z)=0$ and
$$
\HP_{\Omega^{n+1}_{R_{(\nu+1)\bbX}/K}}(z)=
\HP_{\nu\bbX}(z)= s\tbinom{\nu+n-1}{n}.
$$
\end{proposition}

Also, the above lower bound for $\HP_{\Omega^{n+1}_{R_\bbW/K}}(z)$
is attained for a fat point scheme whose support is contained in
a hyperplane (see \cite[Proposition~5.1]{KLL2018}) and the upper bound
for the regularity index of $\Omega^k_{R_\bbW/K}$ is sharp as well
(see \cite[Example~4.3]{KLL2018}).
Based on the isomorphism of graded $R_\bbW$-modules
$\Omega^{n+1}_{R_\bbW/K} \cong (S/\partial I_\bbW)(-n-1)$,
we obtain from Propositions~\ref{PropS1_6} and~\ref{PropS2_1}
the following consequence.

\begin{corollary}\label{CorS2_2}
Let $\bbX=P_1+\cdots+P_s$ be a set of $s$ distinct points
in~$\mathbb{P}^n$, and let $\nu \ge 1$.
There exists $i_0\in \mathbb{N}$ such that for all $i\ge i_0$
we have $(\partial I_\bbX)_i= \mathfrak{M}_i$ and
$$
(\partial I_{\bbX}^{\nu+1})_i = (\partial I_{(\nu+1)\bbX})_i
=(I_{\nu\bbX})_i = (I^\nu_\bbX)_i.
$$
\end{corollary}

For a set of distinct points~$\bbX$ and a subset~$\bbY$ of~$\bbX$,
the vanishing ideals~$I_\bbX$ and $I_\bbY$ satisfy the following
relation.

\begin{lemma} \label{LemS2_4}
Let $\bbX=P_1+\cdots+P_s$ be a set of $s\ge 2$
distinct points in~$\mathbb{P}^n$, let $\bbY$
be a non-empty subset of $\bbX$, and let $k,\ell \ge 0$.
Then, for $i\gg 0$, we have
$$
(I_\bbX^k\cdot I_\bbY^\ell)_i \subseteq
(\partial(I_\bbX^{k+1}\cdot I_\bbY^\ell))_i.
$$
\end{lemma}

\begin{proof}
Let $\nu := k+\ell$. The claim is equivalent to proving the inclusion
\begin{align}\label{Formula1}
(I_\bbX^{\nu-j}\cdot I_\bbY^j)_i \subseteq
(\partial(I_\bbX^{\nu-j+1}\cdot I_\bbY^j))_i
\end{align}
for $i\gg0$ and $j\in \{0,1,...,\nu\}$.
We proceed by induction on $j$.
By Corollary~\ref{CorS2_2}, we have the equalities
$$
(I_\bbX^\nu)_i =(I_{\nu\bbX})_i= (\partial I_\bbX^{\nu+1})_i
$$
for $i\gg 0$, and hence (\ref{Formula1}) holds true for $j=0$.
Suppose that (\ref{Formula1}) holds true for $j-1\ge 0$, i.e.,
there exists $i_0\in \bbN$ such that for all $i\ge i_0$
we have
\begin{align}\label{Formula2}
(I_\bbX^{\nu-j+1}\cdot I_\bbY^{j-1})_i \subseteq
(\partial(I_\bbX^{\nu-j+2}\cdot I_\bbY^{j-1}))_i.
\end{align}
Using Corollary~\ref{CorS2_2}, we find $i_1\in\bbN$
such that
\begin{align*}
(I_\bbX^{\nu-j})_i =(I_{(\nu-j)\bbX})_i= (\partial I_\bbX^{\nu-j+1})_i
\quad\mbox{and}\quad
(I_\bbY^j)_i =(I_{j\bbY})_i
\end{align*}
for all $i\ge i_1$.
Further, let $r_{j\bbY}$ be the regularity index of $\HF_{j\bbY}$.
Then the ho\-mogeneous ideal $I_{j\bbY}$ can be generated by
homogeneous polynomials of degrees $\le r_{j\bbY}+1$ by
\cite[Proposition~1.1]{GM1984}.
Set $r:=  \max\{i_0,i_1,r_{j\bbY}+1\}$.
The ideal $I_\bbY^j$ can also be generated by
homogeneous polynomials of degrees~$\le r$.
Let $i\ge 2r$, and let $F\cdot G \in  (I_\bbX^{\nu-j}\cdot I_\bbY^j)_i$
with homogeneous polynomials $F \in I_\bbX^{\nu-j}$ and $G\in I_\bbY^j$.
Without loss of generality, we may assume that
$F \in (I_\bbX^{\nu-j})_{i-r}$ and $G\in (I_\bbY^j)_{r}$.
Since $i-r \ge 2r-r= r \ge i_1$,
this implies $F \in  (\partial I_\bbX^{\nu-j+1})_{i-r}$.
We may write
$$
F = \sum_{p=0}^n\sum_{q=1}^u H_{pq} \frac{\partial F_{pq}}{\partial X_p}
$$
where $F_{pq} \in (I_\bbX^{\nu-j+1})_{\deg(F_{pq})}$ and
$H_{pq} \in S_{i-r-\deg(F_{pq})+1}$
for all $0\le p\le n$ and $1\le q\le u$. Hence we have
\begin{align*}
F\cdot G &= \sum_{p=0}^n\sum_{q=1}^u H_{pq}
G\frac{\partial F_{pq}}{\partial X_p}
= \sum_{p=0}^n\sum_{q=1}^u H_{pq}
(\frac{\partial(F_{pq}G)}{\partial X_p}-
F_{pq}\frac{\partial G}{\partial X_p})\\
&= \sum_{p=0}^n\sum_{q=1}^u H_{pq}
\frac{\partial(F_{pq}G)}{\partial X_p} -
\sum_{p=0}^n\sum_{q=1}^u H_{pq}F_{pq}
\frac{\partial G}{\partial X_p}.
\end{align*}
Clearly, we have $\frac{\partial(F_{pq}G)}{\partial X_p}\in
\partial(I_\bbX^{\nu-j+1}\cdot I_\bbY^j)$ and thus
$H_{pq}\frac{\partial(F_{pq}G)}{\partial X_p}\in
(\partial(I_\bbX^{\nu-j+1}\cdot I_\bbY^j))_i$.
Moreover, from the inclusion $\partial I_\bbY^j \subseteq I_\bbY^{j-1}$
we deduce $F_{pq}\frac{\partial G}{\partial X_p} \in
I_\bbX^{\nu-j+1}\cdot I_\bbY^{j-1}$.
Also, the inclusion $I_{\bbX}\subseteq I_{\bbY}$ yields
\begin{align}\label{Formula3}
I_\bbX^{\nu-j+2}\cdot I_\bbY^{j-1} \subseteq
I_\bbX^{\nu-j+1}\cdot I_\bbY^{j}.
\end{align}
Hence, from (\ref{Formula2}) and (\ref{Formula3}) we get
$$
H_{pq}F_{pq}\frac{\partial G}{\partial X_p} \in
(I_\bbX^{\nu-j+1}\cdot I_\bbY^{j-1})_i \subseteq
(\partial(I_\bbX^{\nu-j+2}\cdot I_\bbY^{j-1}))_i
\subseteq (\partial(I_\bbX^{\nu-j+1}\cdot I_\bbY^{j}))_i
$$
for all $p,q$. This implies
$F\cdot G\in (\partial(I_\bbX^{\nu-j+1}\cdot I_\bbY^j))_i$.
Since $(I_\bbX^{\nu-j}\cdot I_\bbY^j)_i$ is a $K$-vector space
generated by elements of the form $F\cdot G$, we obtain
the desired inclusion
$$
(I_\bbX^{\nu-j}\cdot I_\bbY^j)_i \subseteq
(\partial(I_\bbX^{\nu-j+1}\cdot I_\bbY^j))_i.
$$
\qed\end{proof}

Notice that the assumtion that~$\bbY$ is a subset of~$\bbX$
is essential for this lemma to hold, as the following
example shows.

\begin{example}
Let $K=\bbQ$ and let $\bbX,\bbY$ be two complete intersections
in $\bbP^2$ given by $\bbX= P_1+P_2+P_3+P_4$ and $\bbY=P_1+P_2+P_5+P_6$,
where $P_1=(1:0:0)$, $P_2=(1:0:1)$, $P_3=(1:1:0)$, $P_4=(1:1:1)$,
$P_5=(1:2:0)$ and $P_6=(1:2:1)$. Then, in $S=K[X_0,X_1,X_2]$, we have
$$
I_\bbX =\langle X_0X_1 -X_1^2, X_0X_2 -X_2^2 \rangle
\quad \mbox{and}\quad
I_\bbY =\langle 2X_0X_1 -X_1^2, X_0X_2 -X_2^2 \rangle.
$$
For $k=\ell=1$, we see that
$$
\HF_{S/(I_\bbX\cdot I_\bbY)}:\ 1\ 3\ 6\ 10\ 11\ 10\ 10\cdots
\quad\mbox{and}\quad
\HF_{S/\partial(I_\bbX^2\cdot I_\bbY)}:\ 1\ 3\ 6\ 10\ 15\ 8\ 8\cdots.
$$
Since  $\HF_{S/(I_\bbX\cdot I_\bbY)}(4)=11<15=
\HF_{S/\partial(I_\bbX^2\cdot I_\bbY)}(4)$, we see that 
$I_\bbX\cdot I_\bbY \nsubseteq \partial(I_\bbX^2\cdot I_\bbY)$.
Thus the formula of the lemma does not hold in general, when~$\bbY$
is not a subset of~$\bbX$.
\end{example}

The preceding lemma can be generalized as follows.

\begin{proposition}\label{PropS2_5}
Let $t\ge 2$, let $\nu_1,\dots,\nu_t \ge 1$,
and let $\bbY_1 \supsetneq \bbY_2 \supsetneq \cdots \supsetneq \bbY_t$
be a descending chain of finite sets of distinct points in $\bbP^n$.
Then, for $i\gg 0$ we have
$$
(I_{\bbY_1}^{\nu_1}\cdot I_{\bbY_2}^{\nu_2}\cdots I_{\bbY_t}^{\nu_t})_i
\subseteq (\partial(I_{\bbY_1}^{\nu_1+1}\cdot I_{\bbY_2}^{\nu_2}
\cdots I_{\bbY_t}^{\nu_t}))_i.
$$
\end{proposition}

In the proof of this proposition we use the following lemma.

\begin{lemma}\label{LemS2_6}
Let $t\ge 2$, let $\nu_1,\dots,\nu_t \ge 1$,
and let $\bbY_1 \supsetneq \bbY_2 \supsetneq \cdots \supsetneq \bbY_t$
be a descending chain of finite sets of distinct points in $\bbP^n$. If
$$
(I_{\bbY_1}^{\nu_1+1}\cdot I_{\bbY_2}^{\nu_2}\cdots I_{\bbY_t}^{\nu_t})_i
\subseteq (\partial(I_{\bbY_1}^{\nu_1+2}\cdot I_{\bbY_2}^{\nu_2}
\cdots I_{\bbY_t}^{\nu_t}))_i
$$
and
$$
(I_{\bbY_1}^{\nu_1}\cdot I_{\bbY_2}^{\nu_2}\cdots I_{\bbY_t}^{\nu_t})_i
\subseteq (\partial(I_{\bbY_1}^{\nu_1+1}\cdot I_{\bbY_2}^{\nu_2}
\cdots I_{\bbY_t}^{\nu_t}))_i
$$
for $i\gg 0$, then
$$
(I_{\bbY_1}^{\nu_1}\cdot I_{\bbY_2}^{\nu_2}\cdots I_{\bbY_t}^{\nu_t+1})_i
\subseteq (\partial(I_{\bbY_1}^{\nu_1+1}\cdot I_{\bbY_2}^{\nu_2}
\cdots I_{\bbY_t}^{\nu_t+1}))_i
$$
for $i\gg 0$.
\end{lemma}

\begin{proof}
Note that the ideal $I_{\bbY_t}$ can be generated by homogeneous
polynomials of degree $\le r_{\bbY_t}+1$, where $r_{\bbY_t}$
is the regularity index of $\HF_{\bbY_t}$.
Set $J := I_{\bbY_1}^{\nu_1}\cdot I_{\bbY_2}^{\nu_2}\cdots I_{\bbY_t}^{\nu_t}$.
By assumption, there exists an integer $i_0 \ge r_{\bbY_t}+1$ such that
$$
(I_{\bbY_1}\cdot J)_i
\subseteq (\partial(I_{\bbY_1}^{2}\cdot J))_i
\quad\mbox{ and }\quad
J_i \subseteq (\partial(I_{\bbY_1}\cdot J))_i
$$
for every $i\ge i_0$. Now we let $i\ge 2i_0$ and consider
$F \in (J\cdot I_{\bbY_t})_i$ of the form
$F=G\cdot H$ with homogeneous polynomials $G \in J$
and $H\in I_{\bbY_t}$. By the choice of~$i_0$,
we may assume $H\in  (I_{\bbY_t})_{i_0}$ and $G\in J_{i-i_0}$.
Since $i-i_0\ge i_0$, we get
$G \in J_{i-i_0} \subseteq (\partial(I_{\bbY_1}\cdot J))_{i-i_0}$.
There are homogeneous polynomials
$G_{01},G_{02},\dots,G_{nm} \in I_{\bbY_1}\cdot J$
and $H_{01},H_{02},\dots,H_{nm} \in S$ such that
$$
G = \sum_{k=0}^n\sum_{\ell=1}^m H_{k\ell}\frac{\partial G_{k\ell}}{\partial X_k}
$$
For any $k,\ell\ge 0$, we observe that
$\frac{\partial (HG_{k\ell})}{\partial X_k} \in
\partial(I_{\bbY_1}\cdot J\cdot I_{\bbY_t})$.
Moreover, we have
$$
H_{k\ell}G_{k\ell}\frac{\partial H}{\partial X_k}\in
(I_{\bbY_1}\cdot J)_i
\subseteq (\partial(I_{\bbY_1}^{2}\cdot J))_i
\subseteq (\partial(I_{\bbY_1}\cdot J \cdot I_{\bbY_t}))_i.
$$
Here the last inclusion follows from the fact that
$I_{\bbY_1}\subseteq I_{\bbY_t}$.
So, we obtain
$$
F = \sum_{k=0}^n\sum_{\ell=1}^m
H_{k\ell}H\frac{\partial G_{k\ell}}{\partial X_k}
= \sum_{k=0}^n\sum_{\ell=1}^m H_{k\ell}
(\frac{\partial (HG_{k\ell})}{\partial X_k}-
G_{k\ell}\frac{\partial H}{\partial X_k}) \in
(\partial(I_{\bbY_1}\cdot J \cdot I_{\bbY_t}))_i.
$$
Since $(J\cdot I_{\bbY_t})_i$ is a $K$-vector space
generated by such elements $F$, the inclusion
$(J\cdot I_{\bbY_t})_i\subseteq
(\partial(I_{\bbY_1}\cdot J \cdot I_{\bbY_t}))_i$
is completely proved.
\end{proof}

\begin{proof}[of Proposition~\ref{PropS2_5}]
Let us prove the claim by induction on $t$.
For $t=2$, the claim follows from Lemma~\ref{LemS2_4}.
Now we consider the case $t>2$ and assume that
the claim holds true for $t-1$. We let
$\nu_1,\dots,\nu_t \ge 0$ and set
$J := I_{\bbY_2}^{\nu_2}\cdots I_{\bbY_{t-1}}^{\nu_{t-1}}$.
By the induction hypothesis, $(I_{\bbY_1}^k \cdot J)_i \subseteq
(\partial(I_{\bbY_1}^{k+1}\cdot J))_i$ holds true for $i\gg 0$
and for all $k\ge 0$. Thus Lemma~\ref{LemS2_6} yields that
$(I_{\bbY_1}^k \cdot J\cdot I_{\bbY_t})_i \subseteq
(\partial(I_{\bbY_1}^{k+1}\cdot J\cdot I_{\bbY_t}))_i$
holds true for $i\gg 0$ and for all $k\ge 0$.
By applying Lemma~\ref{LemS2_6} again, and by induction on the
power of the ideal $I_{\bbY_t}$, we obtain the inclusion
$(I_{\bbY_1}^{\nu_1}\cdot J\cdot I_{\bbY_t}^{\nu_t})_i
\subseteq (\partial(I_{\bbY_1}^{\nu_1+1}\cdot J
\cdot I_{\bbY_t}^{\nu_t}))_i$ for $i\gg 0$, and the claim follows.
\qed\end{proof}

Now we are ready to state and prove a formula for the
Hilbert polynomial of the module of K{\"a}hler differential
$(n+1)$-forms for an arbitrary fat point scheme~$\bbW$ in~$\bbP^n$.
This gives an affirmative answer to the conjecture made
in~\cite[Conjecture~5.7]{KLL2018}.

\begin{theorem}\label{ThmS2_7}
Let $\bbW=m_1P_1+\cdots+m_sP_s$ be a fat point scheme
in~$\mathbb{P}^n\!$, and let~$\bbY$ be the subscheme
$\bbY =(m_1-1)P_1+\cdots+(m_s-1)P_s$ of $\bbW$.
Then the Hilbert polynomial of $\Omega^{n+1}_{R_\bbW /K}$
is given by
$$
\HP_{\Omega^{n+1}_{R_\bbW /K}}(z) = \HP_{\bbY}(z)
= \sum_{j=1}^s\tbinom{m_j+n-2}{n}.
$$
\end{theorem}

\begin{proof}
Note that if $\bbW$ is an equimultiple fat point scheme
in $\mathbb{P}^n$, i.e., if $m_1=\cdots=m_s=\nu$ for some $\nu\ge 1$,
the claim was proved in \cite[Corollary~5.3]{KLL2018}.
Now we prove the claim for the general case. By reordering the 
indices of the points, we may write $\bbW$ as
$$
\bbW = \nu_1\bbX_1+\nu_2\bbX_2+\cdots+ \nu_t\bbX_t,
$$
where $\bbX_1,\dots,\bbX_t$ are disjoint subsets of points in $\bbX$
with $\bbX=\bbX_1+\cdots+\bbX_t$, and where $1\le \nu_1<\nu_2<\cdots<\nu_t$
for some $t\ge 2$.
For $k=1,\dots,t$, we set $\bbY_k = \bbX_k + \cdots + \bbX_t$.
Then we get a descending chain
$\bbX=\bbY_1\supsetneq \bbY_2 \supsetneq\cdots \supsetneq \bbY_t$
of finite sets of distinct points in~$\bbP^n$ such that
$$
\bbW = \nu_1 \bbY_1+ \sum_{k=2}^t(\nu_k-\nu_{k-1})\bbY_k.
$$
Also, it is obviously true that
$\bbY = (\nu_1-1)\bbY_1+ \sum_{k=2}^t(\nu_k-\nu_{k-1})\bbY_k$.
An application of Proposition~\ref{PropS1_6} shows that,
for $i\gg 0$, we have the equalities
\begin{align}\label{Formula4}
(I_\bbW)_i =
(I_{\bbY_1}^{\nu_1}\cdot I_{\bbY_2}^{\nu_2-\nu_1}
\cdots I_{\bbY_t}^{\nu_t-\nu_{t-1}})_i
\mbox{\ and\ }
(I_\bbY)_i =
(I_{\bbY_1}^{\nu_1-1}\cdot I_{\bbY_2}^{\nu_2-\nu_1}
\cdots I_{\bbY_t}^{\nu_t-\nu_{t-1}})_i
\end{align}
It follows from Lemma~\ref{LemS1_11} that
\begin{align} \label{Formula5}
(\partial I_\bbW)_i \;=\;
(\partial(I_{\bbY_1}^{\nu_1}\cdot I_{\bbY_2}^{\nu_2-\nu_1}
\cdots I_{\bbY_t}^{\nu_t-\nu_{t-1}}))_i
\end{align}
for $i\gg 0$.
Furthermore, Proposition~\ref{PropS2_5} implies the inclusion
\begin{align} \label{Formula6}
(I_{\bbY_1}^{\nu_1-1}\cdot I_{\bbY_2}^{\nu_2-\nu_1}
\cdots I_{\bbY_t}^{\nu_t-\nu_{t-1}})_i \subseteq
(\partial(I_{\bbY_1}^{\nu_1}\cdot I_{\bbY_2}^{\nu_2-\nu_1}
\cdots I_{\bbY_t}^{\nu_t-\nu_{t-1}}))_i
\end{align}
for $i\gg 0$.
From (\ref{Formula4}), (\ref{Formula5}) and (\ref{Formula6}) we get
$$
(I_\bbY)_i=(I_{\bbY_1}^{\nu_1-1}\cdot I_{\bbY_2}^{\nu_2-\nu_1}
\cdots I_{\bbY_t}^{\nu_t-\nu_{t-1}})_i \subseteq
(\partial I_\bbW)_i
$$
for $i\gg 0$.
Since $\Omega^{n+1}_{R_\bbW /K} \cong (S/\partial I_\bbW)(-n-1)$,
the Hilbert polynomial of $\Omega^{n+1}_{R_\bbW /K}$ satisfies
$$
\HP_{\Omega^{n+1}_{R_\bbW /K}}(z) = \HP_{S/\partial I_\bbW}(z)
\le \HP_{S/I_\bbY}(z)=\HP_\bbY(z).
$$
Moreover, Proposition~\ref{PropS2_1} yields that
$$
\HP_{\Omega^{n+1}_{R_\bbW /K}}(z)
\ge \sum_{j=1}^s\tbinom{m_j+n-2}{n} =\HP_\bbY(z).
$$
Therefore the desired equality for
$\HP_{\Omega^{n+1}_{R_\bbW /K}}(z)$ follows.
\qed\end{proof}

%
%

\bigbreak
\section{K\"ahler Differential Modules of Fat Point Schemes
in $\mathbb{P}^2$}\label{sec4}

In this section we apply the previous results to examine
the case of fat point schemes in the projective plane $\bbP^2$.
When a fat point scheme~$\bbW$ of~$\bbP^2$ is equimultiple
and is supported on a non-singular conic, the Hilbert function of
the module of K\"ahler differential $k$-forms of $\bbW$ was
computed in~\cite[Section~6]{KLL2018}). However, in the general case
no such detailed information is available.
In this section we use our preceding results to supplement our 
knowledge about $\Omega^k_{R_\bbW /K}$ with some new information.

\begin{proposition} \label{PropS3_1}
Let $\bbW = m_1P_1+\cdots+m_sP_s$ be a fat point scheme
in $\mathbb{P}^2$. The Hilbert polynomials of the modules
of K\"ahler differentials of $\bbW$ are given by
\begin{align*}
\HP_{\Omega^1_{R_\bbW /K}}(z) &=
\tsum{j=1}{s} \;\tfrac{1}{2}\, (3m_j-2)(m_j+1),\\
\HP_{\Omega^2_{R_\bbW/K}}(z) &=
\tsum{j=1}{s} \;\tfrac{1}{2}\, (3m_j+2)(m_j-1),\\
\HP_{\Omega^3_{R_\bbW/K}}(z) &=
\tsum{j=1}{s} \;\tfrac{1}{2} \, m_j(m_j-1).
\end{align*}
\end{proposition}

\begin{proof}
Recall that, the \textit{first fattening}
of $\bbW= m_1P_1+\cdots+m_sP_s$
is the fat point scheme
$\bbW^{(1)} := (m_1+1)P_1+\cdots+(m_s+1)P_s$.
By \cite[Theorem~1.7]{KLL2015}, there is a short exact sequence
of graded $R_\bbW$-modules
\begin{align*}
0\longrightarrow I_\bbW/I_{\bbW^{(1)}}\longrightarrow R_\bbW^3(-1)
\longrightarrow \Omega^1_{R_\bbW/K}\longrightarrow 0,
\end{align*}
and hence we get
\begin{align*}
\HP_{\Omega^1_{R_\bbW/K}}(z)
&= 4\cdot \tsum{j=1}{s} {\textstyle\binom{m_j+1}{2}} -
\tsum{j=1}{s} {\textstyle\binom{m_j+2}{2}} \\
&= \tsum{j=1}{s}  \tfrac{1}{2}\,(4m_j(m_j+1)-(m_j+1)(m_j+2))\\
&= \tsum{j=1}{s}  \tfrac{1}{2}\,(3m_j-2)(m_j+1).
\end{align*}
Also, Theorem~\ref{ThmS2_7} yields
$\HP_{\Omega^3_{R_\bbW/K}}(z)= \sum_{j=1}^s \frac{1}{2} \,m_j(m_j-1)$.
On the other hand, by \cite[Proposition~2.4]{KLL2018}, we have
an exact sequence of graded $R_\bbW$-modules
\begin{align*}
0\longrightarrow \Omega^3_{R_\bbW/K}\longrightarrow
\Omega^2_{R_\bbW/K}\longrightarrow \Omega^1_{R_\bbW/K}
\longrightarrow \mathfrak{m}_{{}_\bbW} \longrightarrow 0
\end{align*}
where $\mathfrak{m}_{{}_\bbW}$ is the homogeneous maximal ideal
of~$R_\bbW$. Thus it follows that
\begin{align*}
\HP_{\Omega^2_{R_\bbW/K}}(z) &=
\HP_{\Omega^3_{R_\bbW/K}}(z)+ \HP_{\Omega^1_{R_\bbW/K}}(z)- \HP_{\mathfrak{m}_\bbW}(z)\\
& = \tsum{j=1}{s} \;\tfrac{1}{2} \, m_j(m_j-1)+ \tsum{j=1}{s} \;\tfrac{1}{2} (3m_j-2)(m_j+1) - 
\tsum{j=1}{s}  \;\tfrac{1}{2}\, m_j(m_j+1)\\
& = \tsum{j=1}{s} \;\tfrac{1}{2}\, (3m_j+2)(m_j-1). \tag*{$\square$}
\end{align*}
\end{proof}

The next remark recalls some information about the degree from where on
we now know the Hilbert function of $\Omega^k_{R_\bbW/K}$.

\begin{remark} \label{RemS3_2}
Set $t := \max\{r_\bbW+1, r_{\bbW^{(1)}}\}$.
The regularity indices of the modules of K\"ahler differentials
of~$\bbW\subseteq \bbP^2$ are bounded by
$$
\ri(\Omega^1_{R_\bbW/K})\le t    \quad\mbox{and}\quad
\ri(\Omega^2_{R_\bbW/K})\le t+1  \quad\mbox{and}\quad
\ri(\Omega^3_{R_\bbW/K})\le t+1.
$$
Moreover, if the support of~$\bbW$ lies on a non-singular conic, 
we have $\ri(\Omega^1_{R_\bbW/K})= t=r_{\bbW^{(1)}}$
(see \cite[Theorem~6.2]{KLL2018}), and the bounds for 
$\ri(\Omega^2_{R_\bbW/K})$ and $\ri(\Omega^3_{R_\bbW/K})$ 
are sharp.
For instance, the scheme $\bbW = P_1+P_2+P_3$ consisting of 
three non-collinear points in $\bbP^2$ satisfies 
$\ri(\Omega^1_{R_\bbW/K})= t=3$ and 
$\ri(\Omega^2_{R_\bbW/K})=\ri(\Omega^3_{R_\bbW/K})= t+1=4$.
\end{remark}

Next, let us look more closely at the module of K\"ahler differential 2-forms
of fat point schemes $\bbW$ in~$\bbP^2$. In general, for a
fat point scheme $\bbW = m_1P_1+\cdots+m_sP_s$ in~$\bbP^n$
with the $i$-th fattening $\bbW^{(i)} = (m_1+i)P_1+\cdots+(m_s+i)P_s$ 
for $i\ge 1$,
\cite[Proposition~5.4]{KLL2018}
implies that the sequence of graded $R_\bbW$-modules
\begin{align}\label{Formula7}
0\rightarrow I_{\bbW^{(1)}}/I_{\bbW^{(2)}}
\stackrel{\alpha}{\rightarrow}
I_\bbW\Omega^1_{S/K}/I_{\bbW^{(1)}}\Omega^1_{S/K}
\stackrel{\beta}{\rightarrow}
\Omega^2_{S/K}/ I_\bbW\Omega^2_{S/K}
\stackrel{\gamma}{\longrightarrow}
\Omega^2_{R_\bbW/K}\rightarrow 0
\end{align}
is a complex. Here the map~$\alpha$ is given by
$\alpha(F+I_{\bbW^{(2)}})=dF+I_{\bbW^{(1)}}\Omega^1_{S/K}$,
the map~$\beta$ is given by
$\beta(GdX_i+ I_{\bbW^{(1)}}\Omega^1_{S/K})= d(GdX_i)+I_\bbW\Omega^2_{S/K}$,
and the map~$\gamma$ is given by $\gamma(H+I_\bbW\Omega^2_{S/K})
=H+(I_\bbW\Omega^2_{S/K}+dI_\bbW\Omega^1_{S/K})$.
In addition, we have $\im(\beta)=\Ker(\gamma)$.
For a fat point scheme~$\bbW$ in~$\bbP^2$,
the complex (\ref{Formula7}) satisfies the following exactness property
which generalizes the case of an equimultiple fat point scheme 
studied in \cite[Proposition~5.5]{KLL2018}.

\begin{proposition} \label{PropS3_3}
Let $\mathbb{W}=m_1P_1+\cdots+m_sP_s$ be a fat point scheme
in~$\mathbb{P}^2$, let
$t :=  \max\{r_{\bbW^{(2)}}, r_{\bbW^{(1)}}+1,r_\bbW+2\}$,
and let $\alpha, \beta$ and $\gamma$ be the maps defined above.
Then, for all $i\ge t$, we have the exact sequence of $K$-vector spaces
$$
\begin{aligned}
0\rightarrow (I_{\bbW^{(1)}}/I_{\bbW^{(2)}})_i
&\stackrel{\alpha}{\longrightarrow}
(I_\bbW\Omega^1_{S/K}/I_{\bbW^{(1)}}\Omega^1_{S/K})_i \\
&\stackrel{\beta}{\longrightarrow}
(\Omega^2_{S/K}/ I_\bbW\Omega^2_{S/K})_i
\stackrel{\gamma}{\longrightarrow}
(\Omega^2_{R_\bbW/K})_i\rightarrow 0.
\end{aligned}
$$
\end{proposition}

\begin{proof}
It suffices to show that $\im(\alpha)=\Ker(\beta)$. Equivalently,
it suffices to prove the equality of Hilbert functions
\begin{align}\label{Formula8}
\HF_{\Omega^2_{R_\bbW/K}}(i) +
\HF_{I_\bbW\Omega^1_{S/K}/I_{\bbW^{(1)}}\Omega^1_{S/K}}(i)
=\HF_{I_{\bbW^{(1)}}/I_{\bbW^{(2)}}}(i)
+ \HF_{\Omega^2_{S/K}/ I_\bbW\Omega^2_{S/K}}(i)
\end{align}
for all $i\ge t$. Since $\Omega^1_{S/K}$ is a free $S$-module
with basis $\{dX_0,dX_1,dX_2\}$ and $\Omega^2_{S/K}$  is
a free $S$-module with basis $\{dX_0dX_1, dX_0dX_2,dX_1dX_2\}$,
for $i\in\bbZ$ we have
$$
\HF_{I_\bbW\Omega^1_{S/K}/I_{\bbW^{(1)}}\Omega^1_{S/K}}(i)
= 3\HF_{\bbW^{(1)}}(i-1)-3\HF_\bbW(i-1)
$$
and
$$
\HF_{\Omega^2_{S/K}/ I_\bbW\Omega^2_{S/K}}(i) = 3\HF_\bbW(i-2).
$$
So, the equality (\ref{Formula8}) can be written as
\begin{align}\label{Formula9}
&\HF_{\Omega^2_{R_\bbW/K}}(i) +
3\HF_{\bbW^{(1)}}(i-1)-3\HF_\bbW(i-1)\\
&= \HF_{\bbW^{(2)}}(i) - \HF_{\bbW^{(1)}}(i)
+ 3\HF_\bbW(i-2).\notag
\end{align}
By Proposition~\ref{PropS3_1} and Remark~\ref{RemS3_2},
$\HF_{\Omega^2_{R_\bbW/K}}(i)= \sum_{j=1}^s \frac{1}{2}\, (3m_j+2)(m_j-1)$
for all $i\ge t$.
Hence the equality (\ref{Formula9}) follows from the fact that
\begin{align*}
&\HF_{\Omega^2_{R_\bbW/K}}(i) +
3\HF_{\bbW^{(1)}}(i-1)-3\HF_\bbW(i-1)
- \HF_{\bbW^{(2)}}(i) \\
&\quad + \HF_{\bbW^{(1)}}(i)
- 3\HF_\bbW(i-2)\\
&= \tsum{j=1}{s} \left( \tfrac{1}{2}\,(3m_j+2)(m_j-1)
+ 4\,\tbinom{m_j+2}{2} - 6\,\tbinom{m_j+1}{2} - \tbinom{m_j+3}{2}
\right)\\
&= \tsum{j=1}{s} \tfrac{1}{2}(3m_j{+} 2)(m_j{-}1 )+ 4(m_j{+}1)(m_j{+}2)-
6m_j(m_j{+}1)-(m_j{+}2)(m_j{+}3)\\
&=0
\end{align*}
for all $i\ge t$. Therefore the claim follows.
\qed\end{proof}

Notice that the exactness of the sequence in this proposition 
allows us to compute values of the Hilbert function of $\Omega^2_{R_\bbW/K}$ 
in the corresponding degrees. Unfortunately,
our final example indicates that this exactness property
does not hold for all $i\in \bbZ$, even when the support
of~$\bbW$ is a complete intersection in~$\bbP^2$.

\begin{example}
Let $\mathbb{W}$ be the fat point scheme given in Example~\ref{ExamS1_1}.
A calculation using~ApCoCoA yields
\begin{align*}
&\HF_{I_{\bbW^{(1)}}/ I_{\bbW^{(2)}}}: &&
0 \ \;0 \ \;0 \ \;0 \ \;\,0 \quad \;\,0 \ \;\,0 \ \;\,2 \ \;\,9 \ \,15 \quad 19 \ 23 \ 26
\ 29 \ 30 \quad 31 \ 31 \cdots,\\
&\HF_{I_\bbW\Omega^1_{S/K}/ I_{\bbW^{(1)}}\Omega^1_{S/K}}:
\!\!&&
0 \ \;0 \ \;0 \ \;0 \ \;\,0 \quad \;\,0 \ \;\,0 \ \;\,2 \ \;\,9 \ \,15 \quad 18 \ 21 \ 22 \ 23 \
23 \quad 23 \ 23 \cdots,\\
&\HF_{\Omega^2_{S/K}/ I_\bbW\Omega^2_{S/K}}:
&&  0 \ \;0 \ \;1 \ \;3 \ \;\,6 \quad 10 \ 15 \ 21 \ 26 \ 27 \quad 28 \ 28 
\ 28 \ 28 \ 28 \quad 28 \ 28 \cdots,\\
&\HF_{\Omega^2_{R_\bbW/K}}:&&
0 \ \;0 \ \;3 \ \;9 \ 18 \quad 30 \ 45 \ 57 \ 53 \ 51 \quad 48 \
47 \ 46 \ 46 \ 46 \quad  46 \ 46 \cdots.
\end{align*}
Thus the proposition holds true for $i=0$, for $i=1$, and for $i\ge 15$.
However, the sequence is not exact for $2\le i\le 14$.
\end{example}

\medskip
\begin{acknowledgements}
The authors were supported by the Vietnam National Foundation 
for Science and Technology Development
(NAFOSTED) under grant number 101.04-2019.07.
The second author thanks the University of Passau for its
hospitality and support during part of the preparation of this paper.
The authors thank the referee for his/her careful reading of the paper.
\end{acknowledgements}

\end{document}